\documentclass[a4paper]{amsart}
\usepackage{ifthen}
\usepackage[leqno]{amsmath}
\usepackage{enumerate}
\usepackage{amstext,amsbsy,amsopn,amsthm,amsgen}
\usepackage{amsfonts,amscd,amsxtra,upref}
\usepackage{mathrsfs}
\usepackage{euscript}
\usepackage{amssymb}
\usepackage[OT4]{fontenc}

\swapnumbers
\theoremstyle{plain}
\newtheorem{Thm}{Theorem}[section]
\newtheorem{Lem}[Thm]{Lemma}
\newtheorem{Cor}[Thm]{Corollary}
\newtheorem{Pro}[Thm]{Proposition}
\theoremstyle{definition}
\newtheorem{Def}[Thm]{Definition}
\theoremstyle{remark}
\newtheorem{Rem}[Thm]{Remark}

\newenvironment{cor}[1]{\begin{Cor}\label{cor:#1}}{\end{Cor}}
\newenvironment{dfn}[1]{\begin{Def}\label{def:#1}}{\end{Def}}
\newenvironment{lem}[1]{\begin{Lem}\label{lem:#1}}{\end{Lem}}
\newenvironment{pro}[1]{\begin{Pro}\label{pro:#1}}{\end{Pro}}
\newenvironment{rem}[1]{\begin{Rem}\label{rem:#1}}{\end{Rem}}
\newenvironment{thm}[1]{\begin{Thm}\label{thm:#1}}{\end{Thm}}

\ifthenelse{\isundefined{\texorpdfstring}}{\newcommand{\texorpdfstring}[2]{#1}}{}

\newcommand{\CCC}{\mathbb{C}}
\newcommand{\RRR}{\mathbb{R}}
\newcommand{\AAa}{\CMcal{A}}
\newcommand{\BBb}{\CMcal{B}}
\newcommand{\DDd}{\CMcal{D}}
\newcommand{\EEe}{\CMcal{E}}
\newcommand{\HHh}{\CMcal{H}}
\newcommand{\KKk}{\CMcal{K}}
\newcommand{\NNn}{\CMcal{N}}
\newcommand{\RRr}{\CMcal{R}}
\newcommand{\TTt}{\CMcal{T}}
\newcommand{\Aa}{\mathfrak{A}}
\newcommand{\Bb}{\mathfrak{B}}
\newcommand{\Cc}{\mathfrak{C}}
\newcommand{\Ll}{\mathfrak{L}}
\newcommand{\Mm}{\mathfrak{M}}
\newcommand{\Rr}{\mathfrak{R}}
\newcommand{\Ww}{\mathfrak{W}}
\newcommand{\Xx}{\mathfrak{X}}
\newcommand{\Zz}{\mathfrak{Z}}
\newcommand{\bB}{\mathfrak{b}}
\newcommand{\uU}{\mathfrak{u}}

\newcommand{\dd}{\colon}
\newcommand{\df}{\stackrel{\textup{def}}{=}}
\newcommand{\geqsl}{\geqslant}
\newcommand{\leqsl}{\leqslant}
\newcommand{\scalar}[2]{\langle #1,#2\rangle}
\newcommand{\scalarr}{\langle\cdot,\mathrm{-}\rangle}
\newcommand{\Card}{\operatorname{Card}}
\newcommand{\card}{\operatorname{card}}
\newcommand{\tr}{\operatorname{tr}}
\newcommand{\Aff}{\operatorname{Aff}}
\newcommand{\Cyl}{\operatorname{Cyl}}
\newcommand{\cdotaff}{\pmb{\cdot}}
\newcommand{\minusaff}{\pmb{-}}
\newcommand{\plusaff}{\pmb{+}}
\newcommand{\traff}{\tr_{\textup{{Aff}}}}
\newcommand{\tII}{\textup{II}}

\newcommand{\COR}[1]{Corollary~\textup{\ref{cor:#1}}}
\newcommand{\DEF}[1]{Definition~\textup{\ref{def:#1}}}
\newcommand{\LEM}[1]{Lemma~\textup{\ref{lem:#1}}}
\newcommand{\PRO}[1]{Proposition~\textup{\ref{pro:#1}}}
\newcommand{\REM}[1]{Remark~\textup{\ref{rem:#1}}}
\newcommand{\THM}[1]{Theorem~\textup{\ref{thm:#1}}}

\begin{document}

\title[Algebra of affiliated operators]
 {Algebra of operators affiliated\\with a finite type~I von Neumann algebra}
\author[P. Niemiec]{Piotr Niemiec}
\address{P. Niemiec{}\\Instytut Matematyki\\{}Wydzia\l{} Matematyki i~Informatyki\\{}
 Uniwersytet Jagiello\'{n}ski\\{}ul. \L{}ojasiewicza 6\\{}30-348 Krak\'{o}w\\{}Poland}
\email{piotr.niemiec@uj.edu.pl}
\thanks{The first author gratefully acknowledges the assistance of the Polish Ministry
 of Sciences and Higher Education grant NN201~546438 for the years 2010--2013.}
\author[A. Wegert]{Adam Wegert}
\address{A. Wegert{}\\Instytut Matematyki\\{}Wydzia\l{} Matematyki, Fizyki i Chemii\\{}
 Uniwersytet \'{S}l\k{a}ski\\{}ul. Bankowa 14\\{}40-007 Katowice\\{}Poland}
\email{a\_wegert@o2.pl}
\begin{abstract}
It is shown that the $*$-algebra of all (closed densely defined linear) operators affiliated with
a finite type~I von Neumann algebra admits a unique center-valued trace, which turns out to be,
in a sense, normal. It is also demonstrated that for no other von Neumann algebras similar
constructions can be performed.
\end{abstract}
\maketitle

\section{Introduction}

With every von Neumann algebra $\Aa$ one can associate the set $\Aff(\Aa)$ of operators (unbounded,
in general) which are \textit{affiliated} with $\Aa$. In \cite{mvn} Murray and von Neumann
discovered that, surprisingly, $\Aff(\Aa)$ turns out to be a unital $*$-algebra when $\Aa$ is
finite. This was in fact the first example of a rich set of \textit{unbounded} operators in which
one can define algebraic binary operations in a natural manner. This concept was later adapted
by Segal \cite{se1,se2} who distinguished a certain class of unbounded operators (namely, measurable
with respect to a fixed normal faithful semi-finite trace) affiliated with an \textit{arbitrary}
semi-finite von Neumann algebra and equipped it with a structure of a $*$-algebra (for
an alternative proof see e.g. \cite{nel} or \S2 in Chapter~IX of \cite{ta2}). A more detailed
investigations of algebras of the form $\Aff(\Aa)$ were initiated by a work of Stone \cite{sto} who
described their models for commutative $\Aa$ in terms of unbounded continuous functions densely
defined on the Gelfand spectrum $\Xx$ of $\Aa$. Much later Kadison \cite{kad} studied this
one-to-one correspondence between operators in $\Aff(\Aa)$ and functions on $\Xx$. Recently Liu
\cite{liu} established an interesting property of $\Aff(\Aa)$ concerning Heisenberg uncertainty
principle. Namely, she showed that the canonical commutation relation, which has the form $AB - BA =
I$, fails to hold for any $A, B \in \Aff(\Aa)$ provided $\Aa$ is finite. We shall obtain her result
for finite \textit{type~I} algebras as a simple corollary of our results. The main purpose
of the paper is to show that whenever $\Aa$ is a finite type~I von Neumann algebra, then
$\Aff(\Aa)$ has a uniquely determined center-valued trace, as shown by

\begin{thm}{main}
Let $\Aa$ be a finite type~I von Neumann algebra and let $\Aff(\Aa)$ be the $*$-algebra of all
operators affiliated with $\Aa$. Then there is a unique linear map $\traff\dd \Aff(\Aa) \to
\Zz(\Aff(\Aa))$ such that
\begin{enumerate}[\upshape(tr1)]
\item $\traff(A)$ is non-negative provided $A \in \Aff(\Aa)$ is so;
\item $\traff(X \cdotaff Y) = \traff(Y \cdotaff X)$ for any $X, Y \in \Aff(\Aa)$;
\item $\traff(Z) = Z$ for each $Z \in \Zz(\Aff(\Aa))$.
\end{enumerate}
What is more,
\begin{equation}\label{eqn:Z}
\Zz(\Aff(\Aa)) = \Aff(\Zz(\Aa))
\end{equation}
and
\begin{enumerate}[\upshape(tr1)]\addtocounter{enumi}{3}
\item $\traff(A) \neq 0$ provided $A \in \Aff(\Aa)$ is non-zero and non-negative;
\item $\traff(X \cdotaff Z) = \traff(X) \cdotaff Z$ for any $X \in \Aff(\Aa)$ and $Z \in
 \Zz(\Aff(\Aa))$;
\item every increasing net $(A_{\sigma})_{\sigma\in\Sigma}$ of self-adjoint members of $\Aff(\Aa)$
 which is majorized by a self-adjoint operator in $\Aff(\Aa)$ has its least upper bound
 in $\Aff(\Aa)$, and
 \begin{equation}\label{eqn:sup}
 \sup_{\sigma\in\Sigma} \traff(A_{\sigma}) = \traff\Bigl(\sup_{\sigma\in\Sigma} A_{\sigma}\Bigr).
 \end{equation}
\end{enumerate}
\end{thm}

It is worth noting that \eqref{eqn:sup} is a natural counterpart of normality (in the terminology
of Takesaki---see Definition~2.1 in Chapter~V of \cite{ta1}) of center-valued traces in finite
von Neumann algebras. It is natural to ask whether the above result may be generalised to a wider
class of von Neumann algebras (e.g. for all finite). Our second goal is to show that the answer is
negative, which is somewhat surprising. A precise formulation of the result is stated below.
We recall that, in general, the set $\Aff(\Aa)$ admits no structure of a vector space, nevertheless,
it is \textit{always} homogeneous and for any $T \in \Aff(\Aa)$ and $S \in \Aa$ the operator $T + S$
is a (well defined) member of $\Aff(\Aa)$. Based on this observation, we may formulate our result
as follows.

\begin{pro}{main}
Let $\Aa$ be a von Neumann algebra and let $\Aff(\Aa)$ be the set of all operators affiliated with
$\Aa$. Assume there exists a function $\varphi\dd \Aff(\Aa) \to \Aff(\Aa)$ with the following
properties:
\begin{enumerate}[\upshape(a)]
\item if $A, B \in \Aa$ are such that $\varphi(A) \in \Aa$, then $\varphi(\alpha A + \beta B) =
 \alpha \varphi(A) + \beta \varphi(B)$ for any scalars $\alpha, \beta \in \CCC$;
\item $\varphi(A)$ is non-negative provided $A \in \Aff(\Aa)$ is so;
\item if $A \in \Aff(\Aa)$ and $B \in \Aa$ are non-negative, and $\varphi(B) \in \Aa$, then
 $\varphi(A + B) = \varphi(A) + \varphi(B)$;
\item $\varphi(AB) = \varphi(BA)$ for all $A, B \in \Aa$;
\item $\varphi(Z) = Z$ for each $Z \in \Zz(\Aa)$;
\item if $A$ and $U$ are members of $\Aa$ and $U$ is unitary, then $U^* \varphi(A) U = \varphi(A)$.
\end{enumerate}
Then $\Aa$ is finite and type~I.
\end{pro}

The paper is organized as follows. In the next section we establish an interesting property
of finite type~I von Neumann algebras, which is crucial in this paper, since all other results,
apart from \PRO{main}, are its consequences. Its proof involves measure-theoretic techniques, which
is in contrast to all other parts of the paper, where all arguments are, roughly speaking, intrinsic
and algebraic. In Section~3 we establish most relevant properties of the set $\Aff(\Aa)$ (for
a finite type~I algebra $\Aa$), including a new proof of the fact that $\Aff(\Aa)$ admits
a structure of a $*$-algebra. In the fourth part we introduce the center-valued trace on $\Aff(\Aa)$
and prove all items of \THM{main}, apart from (tr6), which is shown in Section~5, where we establish
also other order properties of $\Aff(\Aa)$. The last, sixth, part contains a proof of \PRO{main}.

\subsection*{Notation and terminology}
In this paper $\Aa$ is reserved to denote an arbitrary von Neumann algebra acting on a (complex)
Hilbert space $\HHh$. All \textit{operators} are linear, closed and densely defined in a Hilbert
space, \textit{projections} are orthogonal and \textit{non-negative} operators are, by definition,
self-adjoint. The algebra of all bounded operators on $\HHh$ is denoted by $\BBb(\HHh)$. An operator
$T$ in $\HHh$ is \textit{affiliated} with $\Aa$ if $U T U^{-1} = T$ for any unitary operator $U$
belonging to the commutant $\Aa'$ of $\Aa$. $\Aff(\Aa)$ stands for the set of all operators
affiliated with $\Aa$. If $\Aa$ is finite, $\Aff(\Aa)$ may naturally be equipped with a structure
of a $*$-algebra (see e.g. \cite{mvn}). In that case we denote binary algebraic operations
in $\Aff(\Aa)$ by `$\plusaff$' (for addition), `$\minusaff$' (for subtraction) and `$\cdotaff$' (for
multiplication). For any ring $\Rr$, $\Zz(\Rr)$ stands for the center of $\Rr$ (that is, $\Zz(\Rr)$
consists of all elements of $\Rr$ which commute with any element of $\Rr$). This mainly applies for
$\Rr = \Aa$, and $\Rr = \Aff(\Aa)$ provided $\Aa$ is finite. For any operator $S$, we use $\DDd(S)$,
$\NNn(S)$ and $\RRr(S)$ to denote, respectively, the domain, the kernel and the range of $S$.
By $|S|$ we denote the operator $(S^*S)^{\frac12}$. For any collection $\{T_s\}_{s \in S}$
of operators, $\bigoplus_{s \in S} T_s$ is understood as an operator with a maximal domain defined
naturally; that is, $\oplus_{s \in S} x_s$ belongs to the domain of $\bigoplus_{s \in S} T_s$
if $x_s \in \DDd(T_s)$ for each $s \in S$, and $\sum_{s \in S} \|T_s x_s\|^2 < \infty$ (and,
of course, $(\bigoplus_{s \in S} T_s) (\oplus_{s \in S} x_s) = \oplus_{s \in S} (T_s x_s)$).
The center-valued trace in a finite von Neumann algebra $\Ww$ is denoted by $\tr_{\Ww}$.
The center-valued trace on the algebra of operators affiliated with a finite type~I von Neumann
algebra will be denoted by $\traff$. All vector spaces are assumed to be over the field $\CCC$
of complex numbers. For two $C^*$-algebras $\Cc_1$ and $\Cc_2$, we write $\Cc_1 \cong \Cc_2$ when
$\Cc_1$ and $\Cc_2$ are $*$-isomorphic. The direct product of a collection $\{\Cc_s\}_{s \in S}$
of $C^*$-algebras is denoted by $\prod_{s \in S} \Cc_s$ and it consists of all systems
$(a_s)_{s \in S}$ with $a_s \in \Cc_s$ and $\sup_{s \in S} \|a_s\| < \infty$ (cf.
Definition~II.8.1.2 in \cite{bla}). By $I$ we denote the identity operator on $\HHh$.

\section{Key result}

As we shall see in the sequel, all our main results depend on the following theorem, whose proof is
the purpose of this section.

\begin{thm}{fin1}
Assume $\Aa$ is finite and type~I. Then for any $T \in \Aa$ the following conditions are
equivalent:
\begin{enumerate}[\upshape(a)]
\item $\|T \xi\| < \|\xi\|$ for each non-zero vector $\xi \in \HHh$;
\item there is a sequence $Z_1,Z_2,\ldots \in \Zz(\Aa)$ of mutually orthogonal projections such
 that $\sum_{n=1}^{\infty} Z_n = I$ and $\|T Z_n\| < 1$ for any $n \geqsl 1$.
\end{enumerate}
\end{thm}

We shall derive the above theorem as a combination of a classical result on classification
of type~I von Neumann algebras and a measure-theoretic result due to Maharam \cite{mah}.\par
For any positive integer $n$, let $M_n$ be the $C^*$-algebra of all $n \times n$ complex matrices.
Whenever $(X,\Mm,\mu)$ is a finite measure space, we use $L^{\infty}(X,\mu,M_n)$ to denote
the $C^*$-algebra of all $M_n$-valued essentially bounded measurable functions on $X$ (a function
$f = [f_{jk}]\dd X \to M_n$ is \textit{measurable} if each of the functions $f_{jk}\dd X \to \CCC$
is measurable; in other words, $L^{\infty}(X,\mu,M_n) \cong L^{\infty}(X,\mu) \bar{\otimes} M_n$).
The following result is well-known and may easily be derived from Theorems~1.22.13 and 2.3.3
in \cite{sak} (cf. also Theorem~6.6.5 in \cite{kr2}).

\begin{thm}{model1}
For every finite type~I von Neumann algebra $\Aa$ there are a collection
$\{(X_j,\Mm_j,\mu_j)\}_{j \in J}$ of probabilistic measure spaces and a corresponding collection
$\{\nu_j\}_{j \in J}$ of positive integers such that
\begin{equation}\label{eqn:iso1}
\Aa \cong \prod_{j \in J} L^{\infty}(X_j,\mu_j,M_{\nu_j}).
\end{equation}
\end{thm}

We need a slight modification of \eqref{eqn:iso1} (see \THM{model2} below). To this end, let us
introduce certain classical measure spaces, which we call \textit{canonical}. Let $\alpha$ be
an infinite cardinal and $S_{\alpha}$ be a fixed set of cardinality $\alpha$. We consider the set
$D_{\alpha} = \{0,1\}^{S_{\alpha}}$ (of all functions from $S_{\alpha}$ to $\{0,1\}$) equipped with
the \textit{product} $\sigma$-algebra $\Mm_{\alpha}$ and the \textit{product} probabilistic measure
$m_{\alpha}$; that is, $\Mm_{\alpha}$ coincides with the $\sigma$-algebra on $D_{\alpha}$ generated
by all sets of the form
\begin{equation}\label{eqn:cyl}
\Cyl(G) \df \{u \in D_{\alpha}\dd\ u\bigr|_F \in G\}
\end{equation}
where $F$ is a finite subset of $S_{\alpha}$ and $G$ is any subset of $\{0,1\}^F$, while
$m_{\alpha}$ is a unique probabilistic measure on $\Mm_{\alpha}$ such that $m_{\alpha}(\Cyl(G)) =
\card(G)/2^{\card(F)}$ for any such sets $G$ and $F$. It is worth noting that when $\alpha$ is
uncountable and $D_{\alpha}$ is considered with the product topology, not every open set
in $D_{\alpha}$ belongs to $\Mm_{\alpha}$. Additionally, we denote by $(D_0,\Mm_0,m_0)$ a unique
probabilistic measure space with $D_0 = \{0\}$. For simplicity, let $\Card_{\infty}$ stand for
the class of all infinite cardinal numbers. We call the measure spaces
$(D_{\alpha},\Mm_{\alpha},m_{\alpha})$ with $\alpha \in \Card_{\infty} \cup \{0\}$
\textit{canonical}. In the sequel we shall apply the following consequence of a deep result due
to Maharam \cite{mah}:

\begin{thm}{mah}
For any probabilistic measure space $(X,\Mm,\mu)$ there is a sequence \textup{(}finite
or not\textup{)} $\alpha_1,\alpha_2,\ldots \in \Card_{\infty} \cup \{0\}$ such that
the $C^*$-algebras $L^{\infty}(X,\mu)$ and $\prod_{n\geqsl1} L^{\infty}(D_{\alpha_n},m_{\alpha_n})$
are $*$-isomorphic.
\end{thm}

The above result is not explicitly stated in \cite{mah}, but may simply be deduced from Theorems~1
and 2 included there.\par
As a consequence of Theorems~\ref{thm:model1} and \ref{thm:mah} (and the fact that
$L^{\infty}(X,\mu,M_n) \cong L^{\infty}(X,\mu) \bar{\otimes} M_n$), we obtain

\begin{thm}{model2}
For every finite type~I von Neumann algebra $\Aa$ there are collections $\{\alpha_j\}_{j \in J}
\subset \Card_{\infty} \cup \{0\}$ and $\{\nu_j\}_{j \in J} \subset \{1,2,\ldots\}$ such that
\begin{equation*}
\Aa \cong \prod_{j \in J} L^{\infty}(D_{\alpha_j},m_{\alpha_j},M_{\nu_j}).
\end{equation*}
\end{thm}

The following simple lemma will also prove useful for us.

\begin{lem}{1-1}
If $T$ is an arbitrary member of $\Aa$, then $\NNn(T) = \{0\}$ iff the mapping
\begin{equation}\label{eqn:1-1}
\Aa \ni X \mapsto TX \in \Aa
\end{equation}
is one-to-one.
\end{lem}
\begin{proof}
If $\NNn(T) = \{0\}$ and $TX = 0$, then $\RRr(X) \subset \NNn(T) = \{0\}$ and hence $X = 0$. For
the converse, let $E\dd \Bb(\RRR) \to \BBb(\HHh)$ be the spectral measure of $|T|$, defined
on the $\sigma$-algebra $\Bb(\RRR)$ of all Borel subsets of $\RRR$. Then $E(\sigma) \in \Aa$ for any
Borel set $\sigma \subset \RRR$. Since $T E(\{0\}) = 0$, we conclude from the injectivity
of \eqref{eqn:1-1} that $E(\{0\}) = 0$ and thus $\NNn(T) = \NNn(|T|) = \{0\}$.
\end{proof}

\begin{pro}{key}
Let $\Aa$ be a finite type~I von Neumann algebra. Let $\{\alpha_j\}_{j \in J}$ and
$\{\nu_j\}_{j \in J}$ be two collections as in the assertion of \THM{model2}. Further, let $\Phi\dd
\Aa \to \prod_{j \in J} L^{\infty}(D_{\alpha_j},m_{\alpha_j},M_{\nu_j})$ be any $*$-isomorphism.
For an arbitrary operator $T$ in $\Aa$ and $(f_j)_{j \in J} \df \Phi(T)$, the following conditions
are equivalent:
\begin{enumerate}[\upshape(a)]
\item $\|T \xi\| < \|\xi\|$ for each non-zero vector $\xi \in \HHh$;
\item for each $j \in J$, the set $\{x \in D_{\alpha_j}\dd\ \|f_j(x)\| < 1\}$ is of full measure
 $m_{\alpha_j}$.
\end{enumerate}
\end{pro}
\begin{proof}
Assume first that (b) holds. Observe that then $\|T\| = \|\Phi(T)\| \leqsl 1$ and $\Phi((I-T^* T)S)
= (1 - \Phi(T)^* \Phi(T)) \Phi(S) \neq 0$ for any non-zero operator $S \in \Aa$. So, \LEM{1-1}
ensures us that $I-T^*T$ is one-to-one. Consequently, (a) is fulfilled.\par
Now assume that (a) is satisfied, or, equivalently, that $I - T^* T \geqsl 0$ and the mapping $\Aa
\ni X \mapsto (I-T^* T) X \in \Aa$ is one-to-one. This means that
\begin{equation}\label{eqn:norm1}
1-\Phi(T)^* \Phi(T) \geqsl 0
\end{equation}
and $(1-\Phi(T)^* \Phi(T)) g \neq 0$ for each non-zero $g \in \Ll \df \prod_{j \in J}
L^{\infty}(D_{\alpha_j},m_{\alpha_j},M_{\nu_j})$. Suppose, on the contrary, that
\begin{equation}\label{eqn:fk}
m_{\alpha_k}(\{x \in D_{\alpha_k}\dd\ \|f_k(x)\| < 1\}) < 1
\end{equation}
for some $k \in J$. To get a contradiction, it is enough to find a bounded measurable function $u\dd
D_{\alpha_k} \to M_{\nu_k}$ such that $u$ is a non-zero vector
in $L^{\infty}(D_{\alpha_k},m_{\alpha_k},M_{\nu_k})$ and $(1 - f_k^* f_k) u = 0$ (because then
it suffices to put $g_k = u$ and $g_j = 0$ for $j \neq k$ in order to obtain a non-zero vector
$g \df  (g_j)_{j \in J} \in \Ll$ for which $(1 - \Phi(T)^* \Phi(T)) g = 0$). It is now a moment
when we shall make use of the form of the measurable space $(D_{\alpha_k},\Mm_{\alpha_k})$.
If $\alpha_k = 0$, the existence of $u$ is trivial. We therefore assume that $\alpha_k$ is infinite.
Since $f_k\dd \{0,1\}^{S_{\alpha_k}} \to M_{\nu_k}$ is measurable and $\Mm_{\alpha_k}$ is
the product $\sigma$-algebra, we conclude that there exist a \textbf{countable} infinite set
$F \subset S_{\alpha_k}$ and a measurable function $f\dd \{0,1\}^F \to M_{\nu_k}$ such that
\begin{equation}\label{eqn:f}
f_k(\eta) = f(\eta\bigr|_F)
\end{equation}
for any $\eta \in \{0,1\}^{S_{\alpha_k}}$. For simplicity, we put $\Omega = \{0,1\}^F$, $\Mm =
\{G \subset \Omega\dd\ \Cyl(G) \in \Mm_{\alpha_k}\}$ (cf. \eqref{eqn:cyl}) and a measure $\lambda\dd
\Mm \to [0,1]$ by $\lambda(G) = m_{\alpha_k}(\Cyl(G))$ for any $G \in \Mm$. Note that
the probabilistic measure space $(\Omega,\Mm,\lambda)$ is naturally \textit{isomorphic}
to $(D_{\aleph_0},\Mm_{\aleph_0},m_{\aleph_0})$ and hence it is a standard measure space (which is
relevant for us). It suffices to find a measurable bounded function $v\dd \Omega \to M_n$ such that
$v$ is a non-zero vector in $L^{\infty}(\Omega,\lambda,M_n)$ and
\begin{equation}\label{eqn:v}
(1 - f^* f) v \equiv 0
\end{equation}
(because then $u$ may be defined by $u(\eta) \df v(\eta\bigr|_F)$). Let $G \df \{\omega \in
\Omega\dd\ \|f(\omega)\| = 1\} (\in \Mm)$. It follows from \eqref{eqn:norm1}, \eqref{eqn:fk} and
\eqref{eqn:f} that
\begin{equation}\label{eqn:>0}
\lambda(G) > 0.
\end{equation}
Since we deal with (finite-dimensional) matrices, we see that
\begin{equation}\label{eqn:ker}
\forall \omega \in G\dd\ \NNn(I - f(\omega)^* f(\omega)) \neq \{0\}.
\end{equation}
Now we consider a multifunction $\Psi$ on $\Omega$ which assigns to each $\omega \in \Omega$
the kernel of $I - f(\omega)^* f(\omega)$. Equipping the set of all linear subspaces of $\CCC^n$
with the Effros-Borel structure (see \cite{ef1,ef2} or \S6 in Chapter~V in \cite{ta1} and Appendix
there), we conclude that $\Psi$ is measurable (this is a kind of folklore; it may also be simply
deduced e.g. from a combination of Proposition~2.4 in \cite{ern} and Corollary~A.18 in \cite{ta1}).
So, it follows from Effros' theory that there exist measurable functions $h_1,h_2,\ldots\dd \Omega
\to \CCC^n$ such that the set $\{h_k(\omega)\dd\ k \geqsl 1\}$ is a dense subset of $\Psi(\omega)$
for each $\omega \in \Omega$ (to convince of that, consult e.g. subsection~A.16 of Appendix
in \cite{ta2}). We infer from \eqref{eqn:>0} that there is $k \geqsl 1$ such that the set $D \df
\{\omega \in G\dd\ h_k(\omega) \neq 0\}$ has positive measure $\lambda$. Finally, we define $v\dd
\Omega \to M_n$ as follows: for $\omega \in D$, $v(\omega)$ is the matrix which corresponds
(in the canonical basis of $\CCC^n$) to a linear operator
\begin{equation*}
\CCC^n \ni \xi \mapsto \frac{\scalar{\xi}{h_k(\omega)}}{\scalar{h_k(\omega)}{h_k(\omega)}}
h_k(\omega) \in \CCC^n
\end{equation*}
(where $\scalarr$ denotes the standard inner product in $\CCC^n$), and $v(\omega) = 0$ otherwise.
It is readily seen that $v$ is measurable and bounded. What is more, since $\lambda(D) > 0$, we see
that $v$ is a non-zero element of $L^{\infty}(\Omega,\lambda,M_n)$. Finally, \eqref{eqn:v} holds,
because $h_k(\omega) \in \Psi(\omega) = \NNn(I - f^*(\omega) f(\omega))$ for each $\omega$. This
completes the proof.
\end{proof}

Now we are ready to give

\begin{proof}[Proof of \THM{fin1}]
It is clear that (b) is followed by (a). Now assume (a) holds. Let collections
$\{\alpha_j\}_{j \in J}$ and $\{\nu_j\}_{j \in J}$ and a $*$-isomorphism
\begin{equation*}
\Phi\dd \Aa \to \prod_{j \in J} L^{\infty}(D_{\alpha_j},m_{\alpha_j},M_{\nu_j})
\end{equation*}
be as in \PRO{key}. Define $(f_j)_{j \in J} \in \Ll \df \prod_{j \in J}
L^{\infty}(D_{\alpha_j},m_{\alpha_j},M_{\nu_j})$ as $\Phi(T)$. We infer from \PRO{key} that for any
$j \in J$,
\begin{equation}\label{eqn:j}
m_{\alpha_j}(\{x \in D_{\alpha_j}\dd\ \|f_j(x)\| < 1\}) = 1.
\end{equation}
We put $W_{j,n} = \{x \in D_{\alpha_j}\dd\ 1 - 2^{1-n} \leqsl \|f_j(x)\| < 1 - 2^{-n}\}$ and let
$w_{j,n} \in L^{\infty}(D_{\alpha_j},m_{\alpha_j},M_{\nu_j})$ be (constantly) equal to the unit
$\nu_j \times \nu_j$ matrix on $W_{j,n}$ and $0$ off $W_{j,n}$. Observe that $w_n \df
(w_{j,n})_{j \in J}$ is a central projection in $\Ll$ and $\sum_{n=1}^{\infty} w_n = 1$ (thanks
to \eqref{eqn:j}). What is more, it follows from the definition of the sets $W_{j,n}$'s that
$\|\Phi(T) w_n\| \leqsl 1 - 2^{-n}$ for any $n \geqsl 1$. Thus, it remains to define $Z_n$
as $\Phi^{-1}(w_n)$ to finish the proof.
\end{proof}

For simplicity, let us introduce

\begin{dfn}{partition}
A \textit{partition} (in $\Aa$) is an arbitrary collection $\{Z_s\}_{s \in S}$ of mutually
orthogonal projections such that $\sum_{s \in S} Z_s = I$ and $Z_s \in \Zz(\Aa)$ for any $s \in S$.
\end{dfn}

In the sequel we shall need a strengthening of \THM{fin1} stated below. Since its proof is a slight
modification of the argument used in the proof of \THM{fin1}, we skip it and leave it to the reader.

\begin{cor}{central}
Let $\Lambda$ be a countable infinite set of indices and $\{a_{\lambda}\dd\ \lambda \in \Lambda\}$
be a set of positive real numbers such that
\begin{equation}\label{eqn:sup1}
\sup\{a_{\lambda}\dd\ \lambda \in \Lambda\} = 1.
\end{equation}
For $T \in \Aa$, the following conditions are equivalent:
\begin{enumerate}[\upshape(a)]
\item $\|T \xi\| < \|\xi\|$ for every non-zero vector $\xi \in \HHh$;
\item there exists a partition $\{Z_{\lambda}\}_{\lambda\in\Lambda}$ such that $\|T Z_{\lambda}\|
 \leqsl a_{\lambda}$ for every $\lambda \in \Lambda$.
\end{enumerate}
\end{cor}

\section{Algebra of affiliated operators}

The aim of this part is to show the following result.

\begin{thm}{aff}
Let $\Aa$ be finite and type~I, and $\{c_{\lambda}\}_{\lambda\in\Lambda}$ be a countable and
\textbf{unbounded} set of positive real numbers. For any operator $T$ in $\HHh$ the following
conditions are equivalent:
\begin{enumerate}[\upshape(a)]
\item $T \in \Aff(\Aa)$;
\item there is $S \in \Aa$ and a partition $\{Z_{\lambda}\}_{\lambda\in\Lambda}$ for which $T =
 \sum_{\lambda\in\Lambda} c_{\lambda} S Z_{\lambda}$.
\end{enumerate}
\end{thm}

To make the above result more precise (and understood), let us introduce the following

\begin{dfn}{series}
Let $\{Z_{\lambda}\}_{\lambda\in\Lambda}$ be a partition and $\{S_{\lambda}\}_{\lambda\in\Lambda}$
be any collection of operators in $\Aa$. An operator $\sum_{\lambda\in\Lambda} S_{\lambda}
Z_{\lambda}$ is defined as follows:
\begin{equation*}
\DDd\Bigl(\sum_{\lambda\in\Lambda} S_{\lambda} Z_{\lambda}\Bigr) \df \Bigl\{\xi \in \HHh\dd\
\sum_{\lambda\in\Lambda} \|S_{\lambda} Z_{\lambda} \xi\|^2 < \infty\Bigr\}
\end{equation*}
and $(\sum_{\lambda\in\Lambda} S_{\lambda} Z_{\lambda}) \xi = \sum_{\lambda\in\Lambda} (S_{\lambda}
Z_{\lambda} \xi)$ (notice that $S_{\lambda} Z_{\lambda} \xi = Z_{\lambda} S_{\lambda} \xi$ and thus
the vectors $S_{\lambda} Z_{\lambda} \xi$, $\lambda \in \Lambda$, are mutually orthogonal).
\end{dfn}

The following simple result will find many applications in the sequel.

\begin{lem}{block}
Let $(Z_{\lambda})_{\lambda\in\Lambda}$ be a partition. Denote by $\HHh_{\lambda}$ the range
of $Z_{\lambda}$. Then there exists a unitary operator $U\dd \HHh \to \bigoplus_{\lambda\in\Lambda}
\HHh_{\lambda}$ such that for any collection $\{S_{\lambda}\}_{\lambda\in\Lambda} \subset \Aa$,
\begin{equation*}
U \Bigl(\sum_{\lambda\in\Lambda} S_{\lambda} Z_{\lambda}\Bigr) U^{-1} =
\bigoplus_{\lambda\in\Lambda} S_{\lambda}\bigr|_{\HHh_{\lambda}}.
\end{equation*}
\end{lem}
\begin{proof}
For each $\xi \in \HHh$, it is enough to define $U \xi$ as $\oplus_{\lambda\in\Lambda} Z_{\lambda}
\xi$.
\end{proof}

Now we list only most basic consequences of \LEM{block}. Below $\{Z_{\lambda}\}_{\lambda\in\Lambda}$
is a partition in $\Aa$, $\{S_{\lambda}\}_{\lambda\in\Lambda}$ is an arbitrary collection
of operators in $\Aa$ and $\EEe$ stands for the linear span of $\bigcup_{\lambda\in\Lambda}
\RRr(Z_{\lambda})$. Notice that $\EEe$ is dense in $\HHh$.
\begin{enumerate}[($\Sigma$1)]
\item $\sum_{\lambda\in\Lambda} S_{\lambda} Z_{\lambda}$ is closed;
\item $\EEe \subset \DDd(\sum_{\lambda\in\Lambda} S_{\lambda} Z_{\lambda})$ and $\EEe$ is a core
 of $\sum_{\lambda\in\Lambda} S_{\lambda} Z_{\lambda}$;
\item $(\sum_{\lambda\in\Lambda} S_{\lambda} Z_{\lambda})^* = \sum_{\lambda\in\Lambda} S_{\lambda}^*
 Z_{\lambda}$;
\item $\sum_{\lambda\in\Lambda} S_{\lambda} Z_{\lambda} \in \Aff(\Aa)$;
\item if $S_{\lambda} = c_{\lambda} S$ with $c_{\lambda}> 0$ for each $\lambda \in \Lambda$, then
 $\sum_{\lambda\in\Lambda} S_{\lambda} Z_{\lambda}$ is self-adjoint (resp. non-negative; normal) iff
 $S$ is so.
\end{enumerate}

In the proof of \THM{aff} we shall make use of a certain transformation which assigns to any closed
densely defined operator a contraction. In the existing literature there are at least two such
transformations. The first was studied e.g. by Kaufman \cite{kau} and it associates with every
closed densely defined operator $T$ the operator $T(I+T^*T)^{-\frac12}$. The second, quite similar
to the first, is the so-called $\bB$-transform introduced in \cite{pn0} and given by $\bB(T) =
T (I+|T|)^{-1}$. We will use the following properties of the latter transform.

\begin{lem}{b}
Let $T$ and $T_s$, $s \in S$, be closed densely defined operators in $\HHh$ and $\HHh_s$,
respectively. Then:
\begin{enumerate}[\upshape($\bB$1)]
\item the $\bB$-transform establishes a one-to-one correspondence between the set of all closed
 densely defined operators in $\HHh$ and the set of all bounded operators $S$ on $\HHh$ such that
 $\|S \xi\| < \|\xi\|$ for each non-zero vector $\xi \in \HHh$; the inverse transform is given
 by $S \mapsto \uU\bB(S) \df S(I-|S|)^{-1}$.
\item $T$ is bounded iff $\|\bB(T)\| < 1$; conversely, if $S \in \BBb(\HHh)$ and $\|S\| < 1$, then
 $\uU\bB(S) \in \BBb(\HHh)$;
\item $T \in \Aff(\Aa) \iff \bB(T) \in \Aa$;
\item $\bB(U T U^{-1}) = U \bB(T) U^{-1}$ for any unitary operator $U\dd \HHh \to \KKk$;
\item $\bB(\bigoplus_{s \in S} T_s) = \bigoplus_{s \in S} \bB(T_s)$.
\end{enumerate}
\end{lem}

Below we use the $\bB$- and $\uU\bB$-transforms also as complex-valued functions defined on $\CCC$,
given by appropriate analogous formulas.

\begin{proof}[Proof of \THM{aff}]
Property ($\Sigma$4) shows that (a) is implied by (b). Now assume that $T \in \Aff(\Aa)$. Then
$\bB(T) \in \Aa$ and $\|\bB(T)\xi\| < \|\xi\|$ for each $\xi \neq 0$ (see ($\bB$1) and ($\bB$3)).
Using \COR{central} with $a_{\lambda} \df \bB(c_{\lambda}) = \frac{c_{\lambda}}{1+c_{\lambda}}$,
we obtain a partition $\{Z_{\lambda}\}_{\lambda\in\Lambda}$ such that $\|\bB(T) Z_{\lambda}\| \leqsl
a_{\lambda} < 1$. We now infer from ($\bB$2) that there exist operators $S_{\lambda} \in
\BBb(\HHh)$, $\lambda \in \Lambda$, such that
\begin{equation}\label{eqn:b(T)}
\bB(S_{\lambda}) = \bB(T) Z_{\lambda}.
\end{equation}
We can express $S_{\lambda}$ directly as $S_{\lambda} = \bB(T) Z_{\lambda} (I - |\bB(T)
Z_{\lambda}|\bigr)^{-1}$ and this formula clearly implies that $S_{\lambda} \in \Aa$. It is
a well-known property of the functional calculus for self-adjoint (bounded) operators that
$\|\uU\bB(A)\| = \uU\bB(\|A\|) = \frac{\|A\|}{1-\|A\|}$ for any non-negative operator $A$ of norm
less than $1$. We shall apply this property for $A = |\bB(T) S_{\lambda}|$. We have:
\begin{equation}\label{eqn:Slamb}
S_{\lambda} Z_{\lambda} = \bB(T) Z_{\lambda}^2 (I - |\bB(T) Z_{\lambda}|)^{-1} = \bB(T) Z_{\lambda}
(I - |\bB(T) Z_{\lambda}|)^{-1} = S_{\lambda}
\end{equation}
and
\begin{align}\label{eqn:clamb}
\|S_{\lambda}\| &= \|\bB(T) Z_{\lambda} (I - |\bB(T) Z_{\lambda}|)^{-1}\|\\
&= \Bigl\| |\bB(T) Z_{\lambda}| (I - |\bB(T) Z_{\lambda}|)^{-1}\Bigr\|\nonumber\\
&= \|\uU\bB(|\bB(T) Z_{\lambda}|)\| = \uU\bB\bigl(\bigl\||\bB(T) Z_{\lambda}|\bigr\|\bigr) \leqsl
\uU\bB(a_{\lambda}) = c_{\lambda}.\nonumber
\end{align}
Define $S \df \sum_{\lambda\in\Lambda} \frac{1}{c_{\lambda}} S_{\lambda}$. From \eqref{eqn:Slamb}
and \eqref{eqn:clamb} we infer that the series converges in the strong operator topology.
Consequently, $S \in \Aa$. Moreover, $c_{\lambda} S Z_{\lambda} = S_{\lambda} Z_{\lambda}$. In order
to prove that $T = \sum_{\lambda\in\Lambda} c_{\lambda} S Z_{\lambda}$, it is enough to show that
$\bB(T) = \bB(\sum_{\lambda\in\Lambda} c_{\lambda} S Z_{\lambda})$. Using \LEM{block}, a unitary
operator $U$ and subspaces $\HHh_{\lambda}$ appearing there, properties ($\bB$4) and ($\bB$5)
formulated in \LEM{b}, and \eqref{eqn:b(T)}, we get
\begin{multline*}
\bB\Bigl(\sum_{\lambda\in\Lambda} c_{\lambda} S Z_{\lambda}\Bigr)
= U^{-1} \Bigl(\bigoplus_{\lambda\in\Lambda}
\bB\bigl(c_{\lambda} S\bigr|_{\HHh_{\lambda}}\bigr)\Bigr) U
= U^{-1} \Bigl(\bigoplus_{\lambda\in\Lambda}
\bB\bigl(S_{\lambda}\bigr|_{\HHh_{\lambda}}\bigr)\Bigr) U\\
= U^{-1} \Bigl(\bigoplus_{\lambda\in\Lambda} \bB(S_{\lambda})\bigr|_{\HHh_{\lambda}}\Bigr) U
= \bigoplus_{\lambda\in\Lambda} \bB(T)\bigr|_{\HHh_{\lambda}} = \bB(T)
\end{multline*}
and we are done.
\end{proof}

As a first application of \THM{aff} we obtain

\begin{cor}{common}
Let $\Aa$ be finite and type~I, and $\Lambda \df \{\nu = (\nu_1,\ldots,\nu_k)\dd\ \nu_1,\ldots,
\nu_k \geqsl 1\}$. For any collection $T_1,\ldots,T_k \in \Aff(\Aa)$ there exist a partition
$\{Z_{\nu}\}_{\nu\in\Lambda}$ in $\Aa$ and operators $S_1,\ldots,S_k \in \Aa$ such that for each
$j \in \{1,\ldots,k\}$,
\begin{equation}\label{eqn:common}
T_j = \sum_{\nu\in\Lambda} \nu_j S_j Z_{\nu}.
\end{equation}
\end{cor}
\begin{proof}
Using \THM{aff}, write each $T_j$ as $\sum_{n=1}^{\infty} n S_j Z^{(j)}_n$ and put $Z_{\nu} =
Z^{(1)}_{\nu_1} \cdot \ldots \cdot Z^{(k)}_{\nu_k}$.
\end{proof}

\begin{rem}{alg}
\COR{common} gives an alternative proof of the Murray-von Neumann theorem \cite{mvn} that
$\Aff(\Aa)$ can naturally be equipped with a structure of a $*$-algebra provided $\Aa$ is finite and
type~I (the assumption that $\Aa$ is type~I is superfluous; however, our proof works only
in that case). Indeed, if $T_1,\ldots,T_k$ are arbitrary members of $\Aff(\Aa)$, and
$\{Z_{\nu}\}_{\nu\in\Lambda}$ and $S_1,\ldots,S_k \in \Aa$ are as in \eqref{eqn:common}, then for
any polynomial $p(x_1,\ldots,x_n)$ in $n$ non-commuting variables we may define $p(T_1,\ldots,T_n)$
as $\sum_{\nu\in\Lambda} p(\nu_1 S_1,\ldots,\nu_k S_k) Z_{\nu}$. With such a definition, the linear
span of $\bigcup_{\nu\in\Lambda} \RRr(Z_{\nu})$ is a core for each operator of the form
$p(T_1,\ldots,T_k)$. Furthermore, the representation \eqref{eqn:common} enables us to prove briefly
that $T_1 = T_2$ provided $T_1 \subset T_2$. It is now easy to conclude from all these observations
that $\Aff(\Aa)$ admits a structure of a $*$-algebra (in particular, all algebraic laws for
an algebra, such as associativity, are satisfied). We leave the details to interested readers.
\end{rem}

\section{Trace}

We now turn to the concept of a center-valued trace on $\Aff(\Aa)$. \textbf{In this section
$\Aa$ is assumed to be finite and type~I.} We recall that `$\plusaff$', `$\minusaff$' and
`$\cdotaff$' denote the binary operations in $\Aff(\Aa)$. Our main goal is to prove all items
of \THM{main}, apart from (tr6), which will be shown in the next section.\par
We begin with

\begin{pro}{center}
$\Aff(\Zz(\Aa)) = \Zz(\Aff(\Aa))$.
\end{pro}
\begin{proof}
Take $T \in \Aff(\Zz(\Aa)) \subset \Aff(\Aa)$. Since $\Zz(\Aa)$ is also finite and type~I,
it follows from \THM{aff} that $T$ has the form $T = \sum_{n=1}^{\infty} n S Z_n$ with $S, Z_n \in
\Zz(\Aa)$. Similarly, any $X \in \Aff(\Aa)$ has the form $X = \sum_{n=1}^{\infty} n Y W_n$ with
$Y \in \Aa$ and $W_n \in \Zz(\Aa)$. Then $SY = YS$ and it follows from \REM{alg} that for $\Lambda =
\{\nu = (\nu_1,\nu_2)\dd\ \nu_1,\nu_2 \geqsl 1\}$ and $Z_{\nu} = Z_{\nu_1} W_{\nu_2}$, $T \cdotaff X
= \sum_{\nu\in\Lambda} \nu_1 \nu_2 S Y Z_{\nu} = \sum_{\nu\in\Lambda} \nu_2 \nu_1 Y S Z_{\nu} = X
\cdotaff T$, which shows that $T \in \Zz(\Aff(\Aa))$. In particular, $\Zz(\Aa) \subset
\Zz(\Aff(\Aa))$.\par
Conversely, take $T \in \Zz(\Aff(\Aa))$ of the form $T = \sum_{n=1}^{\infty} n S Z_n$ (with $S \in
\Aa$ and $Z_n \in \Zz(\Aa)$). Then $S Z_n = \frac1n T \cdotaff Z_n$ belongs to $\Zz(\Aff(\Aa)) \cap
\Aa \subset \Zz(\Aa)$ and thus $S = \sum_{n=1}^{\infty} S Z_n \in \Zz(\Aa)$. Another application
of \THM{aff} (for the von Neumann algebra $\Zz(\Aa)$) yields $T \in \Aff(\Zz(\Aa))$.
\end{proof}

\begin{lem}{valid}
Let $\{Z_{\lambda}\}_{\lambda\in\Lambda}$ and $\{W_{\gamma}\}_{\gamma\in\Gamma}$ be two partitions
in $\Aa$ and let $\{T_{\lambda}\}_{\lambda\in\Lambda}$ and $\{S_{\gamma}\}_{\gamma\in\Gamma}$ be
two collections of operators in $\Aa$ such that
\begin{equation}\label{eqn:2repr}
\sum_{\lambda\in\Lambda} T_{\lambda} Z_{\lambda} = \sum_{\gamma\in\Gamma} S_{\gamma} W_{\gamma}.
\end{equation}
Then
\begin{equation*}
\sum_{\lambda\in\Lambda} \tr_{\Aa}(T_{\lambda}) Z_{\lambda} = \sum_{\gamma\in\Gamma}
\tr_{\Aa}(S_{\gamma}) W_{\gamma}.
\end{equation*}
\end{lem}
\begin{proof}
For $P_{\lambda,\gamma} \df Z_{\lambda} W_{\gamma}$, we have, by \eqref{eqn:2repr}, $T_{\lambda}
P_{\lambda,\gamma} = S_{\gamma} P_{\lambda,\gamma}$ for any $\lambda \in \Lambda$ and $\gamma \in
\Gamma$. Consequently, $\tr_{\Aa}(T_{\lambda}) P_{\lambda,\gamma} = \tr_{\Aa}(T_{\lambda}
P_{\lambda,\gamma}) = \tr_{\Aa}(S_{\gamma} P_{\lambda,\gamma}) = \tr_{\Aa}(S_{\gamma})
P_{\lambda,\gamma}$. So,
\begin{equation*}
\sum_{\lambda\in\Lambda} \tr_{\Aa}(T_{\lambda}) Z_{\lambda} = \sum_{\lambda\in\Lambda}
\sum_{\gamma\in\Gamma} \tr_{\Aa}(T_{\lambda}) P_{\lambda,\gamma} =
\sum_{\gamma\in\Gamma} \sum_{\lambda\in\Lambda} \tr_{\Aa}(S_{\gamma}) P_{\lambda,\gamma} =
\sum_{\gamma\in\Gamma} \tr_{\Aa}(S_{\gamma}) W_{\gamma}
\end{equation*}
and we are done.
\end{proof}

Now we are ready to introduce

\begin{dfn}{trace}
The \textit{center-valued trace} in $\Aff(\Aa)$ is a mapping
\begin{equation*}
\traff\dd \Aff(\Aa) \to \Zz(\Aff(\Aa))
\end{equation*}
defined as follows. For any partition $\{Z_{\lambda}\}_{\lambda\in\Lambda}$ in $\Aa$ and
a collection $\{S_{\lambda}\}_{\lambda\in\Lambda} \subset \Aa$,
\begin{equation*}
\traff\Bigl(\sum_{\lambda\in\Lambda} S_{\lambda} Z_{\lambda}\Bigr) = \sum_{\lambda\in\Lambda}
\tr_{\Aa}(S_{\lambda}) Z_{\lambda}.
\end{equation*}
\THM{aff} and \LEM{valid} ensure us that the definition is full and correct, while \PRO{center}
(and its proof) shows that indeed $\traff(T)$ belongs to $\Zz(\Aff(\Aa))$ for any $T \in \Aff(\Aa)$.
\end{dfn}

For transparency, let us isolate the uniqueness part of \THM{main} in the following

\begin{lem}{uniq}
If $\tr'\dd \Aff(\Aa) \to \Zz(\Aff(\Aa))$ is a linear mapping which satisfies axioms
\textup{(tr1)--(tr3)} \textup{(}with $\traff$ replaced by $\tr'$\textup{)}, then $\tr' = \traff$.
\end{lem}
\begin{proof}
Fix a partition $\{Z_n\}_{n=1}^{\infty}$ in $\Aa$ and consider the map
\begin{equation*}
f\dd \Aa \ni S \mapsto \tr'\Bigl(\sum_{n=1}^{\infty} n S Z_n\Bigr) \cdotaff
\Bigl(\sum_{n=1}^{\infty} \frac1n Z_n\Bigr) \in \Zz(\Aff(\Aa)).
\end{equation*}
It is clear that $f$ is linear. Moreover, for any $S_1, S_2 \in \Aa$, using (tr2), we get
\begin{align*}
f(S_1S_2) &= \tr'\Bigl(\sum_{n=1}^{\infty} n S_1S_2 Z_n\Bigr) \cdotaff \Bigl(\sum_{n=1}^{\infty}
\frac1n Z_n\Bigr)\\
&= \tr'\Bigl(\sum_{n=1}^{\infty} \sqrt{n} S_1 Z_n \cdotaff \sum_{n=1}^{\infty} \sqrt{n} S_2
Z_n\Bigr) \cdotaff \Bigl(\sum_{n=1}^{\infty} \frac1n Z_n\Bigr)\\
&= \tr'\Bigl(\sum_{n=1}^{\infty} \sqrt{n} S_2 Z_n \cdotaff \sum_{n=1}^{\infty} \sqrt{n} S_1
Z_n\Bigr) \cdotaff \Bigl(\sum_{n=1}^{\infty} \frac1n Z_n\Bigr)\\
&= \tr'\Bigl(\sum_{n=1}^{\infty} n S_2S_1 Z_n\Bigr) \cdotaff \Bigl(\sum_{n=1}^{\infty} \frac1n
Z_n\Bigr) = f(S_2S_1).
\end{align*}
Further, if $S \in \Aa$ in non-negative, then $T \df \sum_{n=1}^{\infty} n S Z_n$ is non-negative
as well (by ($\Sigma$5)). Consequently, $\tr'(T)$ is non-negative and therefore so is $f(S)$. Also
for $C \in \Zz(\Aa)$ we have $\sum_{n=1}^{\infty} n C Z_n \in \Aff(\Zz(\Aa)) = \Zz(\Aff(\Aa))$ (cf.
the proof of \PRO{center}), thus, thanks to (tr3),
\begin{equation*}
f(C) = \sum_{n=1}^{\infty} n C Z_n \cdotaff \sum_{n=1}^{\infty} \frac1n Z_n = \sum_{n=1}^{\infty} C
Z_n = C.
\end{equation*}
Finally, we claim that $f(S)$ is bounded for any $S \in \Aa$. (This will imply that $f(\Aa) \subset
\Zz(\Aff(\Aa)) \cap \Aa = \Zz(\Aa)$.) To see tee this, it is enough to assume that $S \in \Aa$ is
non-negative. Then the operator $\|S\| I - S$ is non-negative as well and hence both $f(S)$ and
$f(\|S\|I-S)$ are non-negative. But $f(S) \plusaff f(\|S\|I-S) = f(\|S\|I) = \|S\|I$. Consequenty,
$f(S)$ is bounded, as we claimed.\par
As $f\dd \Aa \to \Zz(\Aa)$ satisfies all axioms of the center-valued trace in $\Aa$ (cf. e.g.
Theorem~8.2.8 in \cite{kr2}), we have $f = \tr_{\Aa}$ and consequently for each $S \in \Aa$:
\begin{multline*}
\tr'\Bigl(\sum_{n=1}^{\infty} n S Z_n\Bigr) = f(S) \cdotaff \sum_{n=1}^{\infty} n Z_n = \tr_{\Aa}(S)
\cdotaff \sum_{n=1}^{\infty} n Z_n = \sum_{n=1}^{\infty} n \tr_{\Aa}(S) Z_n\\
= \traff\Bigl(\sum_{n=1}^{\infty} n S Z_n\Bigr).
\end{multline*}
Since the partition was arbitrary, an application of \THM{aff} completes the proof.
\end{proof}

\begin{proof}[Proof of \THM{main}]
As we announced, property (tr6) shall be established in the next section. The linearity of $\traff$
follows from \COR{common} and the very definition of $\traff$ (see also \REM{alg} and \LEM{valid}).
Conditions (tr1) and (tr4) are immediate consequences of ($\Sigma$5). Property (tr3) follows from
the fact that each $C \in \Zz(\Aff(\Aa))$ may be written in the form $C = \sum_{n=1}^{\infty} n W
Z_n$ where $W \in \Zz(\Aa)$ and $\{Z_n\}_{n=1}^{\infty}$ is a partition (see the proof
of \PRO{center}). Further, (tr2) and (tr5) are implied by suitable properties of $\tr_{\Aa}$ and
the way the multiplication in $\Aff(\Aa)$ is defined (below we use \COR{common} with $\Lambda =
\{\nu = (\nu_1,\nu_2)\dd\ \nu_1,\nu_2 \geqsl 1\}$): if $T = \sum_{\nu\in\Lambda} \nu_1 S Z_{\nu}$
and $X = \sum_{\nu\in\Lambda} \nu_2 Y Z_{\nu}$ (with $Y \in \Zz(\Aa)$ provided $X \in
\Zz(\Aff(\Aa))$), then $T \cdotaff X = \sum_{\nu\in\Lambda} \nu_1 \nu_2 S Y Z_{\nu}$ and hence
\begin{equation*}
\traff(T \cdotaff X) = \sum_{\nu\in\Lambda} \nu_1 \nu_2 \tr_{\Aa}(SY) Z_{\nu}
= \sum_{\nu\in\Lambda} \nu_1 \nu_2 \tr_{\Aa}(YS) Z_{\nu} = \traff(X \cdotaff T);
\end{equation*}
and if $X \in \Zz(\Aff(\Aa))$, we get
\begin{equation*}
\traff(T \cdotaff X) = \sum_{\nu\in\Lambda} \nu_1 \nu_2 \tr_{\Aa}(S)Y Z_{\nu}
= \sum_{\nu\in\Lambda} \nu_1 \tr_{\Aa}(S) Z_{\nu} \cdotaff \sum_{\nu\in\Lambda} \nu_2 Y Z_{\nu}
= \traff(T) \cdotaff X.
\end{equation*}
Finally, uniqueness of $\traff$ was already established in \LEM{uniq} and \eqref{eqn:Z} is just
the assertion of \PRO{center}.
\end{proof}

It is worth noting that $\traff(S) = \tr_{\Aa}(S)$ for each $S \in \Aa$, the proof of which is left
as a simple exercise.\par
As an immediate consequence of \THM{main}, we get the following

\begin{cor}{heisenberg}
Suppose that $\Aa$ is finite and type~I. There are no $X, Y \in \Aff(\Aa)$ such that $X \cdotaff Y
\minusaff Y \cdotaff X = I$.
\end{cor}
\begin{proof}
Apply the trace for both sides.
\end{proof}

The above result for arbitrary finite von Neumann algebras was proved in \cite{liu}.

\section{Ordering}

Throughout this section, $\Aa$ continues to be finite and type~I; and $\scalarr$ stands for
the inner product of $\HHh$. Similarly as in $C^*$-algebras, we may distinguish \textit{real} part
of $\Aff(\Aa)$ and introduce a natural ordering in it. To this end, we introduce

\begin{dfn}{ord}
The \textit{real} part $\Aff_s(\Aa)$ of $\Aff(\Aa)$ is the set of all self-adjoint operators
in $\Aff(\Aa)$. Additionally, we put $\Aa_s = \Aa \cap \Aff_s(\Aa)$. For $A \in \Aff_s(\Aa)$
we write $A \geqsl 0$ if $A$ is non-negative (that is, if $\scalar{A \xi}{\xi} \geqsl 0$ for each
$\xi \in \DDd(A)$; or, equivalently, if the spectrum of $A$ is contained in $[0,\infty)$). For two
operators $A_1, A_2 \in \Aff_s(\Aa)$ we write $A_1 \leqsl A_2$ or $A_2 \geqsl A_1$ if $A_1 \minusaff
A_2 \geqsl 0$.\par
The least upper bound (in $\Aff_s(\Aa)$) of a collection $\{B_s\}_{s \in S} \subset \Aff_s(\Aa)$ is
denoted by $\sup_{s \in S} B_s$ provided it exists.
\end{dfn}

The following simple result gives another description of the ordering defined above.

\begin{lem}{ord}
Let $A$ and $B$ be arbitrary members of $\Aff_s(\Aa)$.
\begin{enumerate}[\upshape(a)]
\item If both $A$ and $B$ are non-negative, so is $A \plusaff B$.
\item The following conditions are equivalent:
 \begin{enumerate}[\upshape(i)]
 \item $A \leqsl B$;
 \item $\scalar{A \xi}{\xi} \leqsl \scalar{B \xi}{\xi}$ for any $\xi \in \DDd(A) \cap \DDd(B)$.
 \end{enumerate}
\end{enumerate}
\end{lem}
\begin{proof}
All properties follow from the fact that $\DDd(A) \cap \DDd(B)$ is a core of self-adjoint operators
$B \minusaff A$ and $A \plusaff B$.
\end{proof}

It is now readily seen that the ordering `$\leqsl$' in $\Aff_s(\Aa)$ is reflexive, transitive and
antisymmetric (which means that $A = B$ provided $A \leqsl B$ and $B \leqsl A$), and that it is
compatible with the linear structure of $\Aff_s(\Aa)$. Another property is established below.

\begin{pro}{commute}
If $A, B \in \Aff_s(\Aa)$ are non-negative and $A \cdotaff B = B \cdotaff A$, then $A \cdotaff B$ is
non-negative as well.
\end{pro}
\begin{proof}
By \COR{common} and ($\Sigma$5), we may express $A$ and $B$ in the forms $A = \sum_{\nu\in\Lambda}
\nu_1 S Z_{\nu}$ and $B = \sum_{\nu\in\Lambda} \nu_2 T Z_{\nu}$ where $S, T \in \Aa$ are
non-negative. Moreover, we know that then
\begin{equation}\label{eqn:aux1}
A \cdotaff B = \sum_{\nu\in\Lambda} \nu_1 \nu_2 ST Z_{\nu}
\end{equation}
and $B \cdotaff A = \sum_{\nu\in\Lambda} \nu_2 \nu_1 TS Z_{\nu}$. We now deduce from these
connections and the assumption that $ST = TS$ and, consequently, $ST \geqsl 0$. Now the assertion
follows from \eqref{eqn:aux1} and ($\Sigma$5).
\end{proof}

For transparency, we isolate a part of (tr6) (in \THM{main}) below.

\begin{lem}{sup}
Let $\AAa = \{A_{\sigma}\}_{\sigma\in\Sigma} \in \Aff_s(\Aa)$ be an increasing net \textup{(}indexed
by a directed set $\Sigma$\textup{)}, bounded above by $A \in \Aff_s(\Aa)$. Then $\AAa$ has a least
upper bound in $\Aff_s(\Aa)$.
\end{lem}
\begin{proof}
First of all, we may and do assume that $A_{\sigma} \geqsl 0$ for any $\sigma \in \Sigma$. (Indeed,
fixing $\sigma_0 \in \Sigma$ and putting $\Sigma' \df \{\sigma \in \Sigma\dd\ \sigma \geqsl
\sigma_0\}$, $\AAa' \df \{A'_{\sigma'}\}_{\sigma'\in\Sigma'}$ with $A_{\sigma'}' \df A_{\sigma}
\minusaff A_{\sigma_0}$ and $A' \df A \minusaff A_{\sigma_0}$, it is easy to verify that $\AAa'$ is
an increasing net of non-negative operators upper bounded by $A'$, and $\sup_{\sigma\in\Sigma}
A_{\sigma} = A_{\sigma_0} \plusaff \sup_{\sigma'\in\Sigma'} A_{\sigma'}'$.) Using \THM{aff} and
($\Sigma$5), we may express $A$ in the form $A = \sum_{n=1}^{\infty} n B Z_n$ where $B \in \Aa$ is
non-negative and $\{Z_n\}_{n=1}^{\infty}$ is a partition in $\Aa$. Fix $k \geqsl 1$. It follows from
\PRO{commute} that the operators $(A \minusaff A_{\sigma}) \cdotaff Z_k$ and $A_{\sigma} \cdotaff
Z_k$ are non-negative for any $\sigma \in \Sigma$. So, $0 \leqsl A_{\sigma} \cdotaff Z_k \leqsl
A \cdotaff Z_k$. Since $A \cdotaff Z_k = k B Z_k$ is bounded, we now conclude (e.g. from \LEM{ord})
that $A_{\sigma} \cdotaff Z_k$ is bounded as well. Moreover, the same argument shows that the net
$\{A_{\sigma} \cdotaff Z_k\}_{\sigma\in\Sigma} \subset \Aa_s$ is increasing and upper bounded
by $A \cdotaff Z_k \in \Aa_s$. From a classical property of von Neumann algebras we infer that this
last net has a least upper bound in $\Aa_s$, say $G_k$. We now put $G \df \sum_{k=1}^{\infty} G_k
Z_k$. Note that $G \in \Aff_s(\Aa)$ (see ($\Sigma$4) and ($\Sigma$5)). Since $0 \leqsl A_{\sigma}
\cdotaff Z_k) \leqsl G_k \leqsl A \cdotaff Z_k = (A \cdotaff Z_k) Z_k$, we have $G_k = G_k Z_k
(= G \cdotaff Z_k)$ and $A_{\sigma} \cdotaff Z_k = (A_{\sigma} \cdotaff Z_k) Z_k$, and thus
$A_{\sigma} \cdotaff Z_k \leqsl G \cdotaff Z_k \leqsl A \cdotaff Z_k$ for any $\sigma \in \Sigma$
and $k \geqsl 1$. These inequalities imply that
\begin{equation}\label{eqn:G}
A_{\sigma} \leqsl G \leqsl A \qquad (\sigma\in\Sigma)
\end{equation}
(because for $X \in \{A_{\sigma},G,A\}$, $X = \sum_{k=1}^{\infty} (X \cdotaff Z_k) Z_k$ in the sense
of \DEF{series}; then apply \LEM{ord}). We shall check that $G = \sup_{\sigma\in\Sigma} A_{\sigma}$.
To this end, take an arbitrary upper bound $A' = \sum_{n=1}^{\infty} n B' Z_n'$ (where $B' \in \Aa$
is self-adjoint) of $\AAa$. It remains to check that $G \leqsl A'$. In what follows, to avoid
misunderstandings, `$\sup^{\Aa}$' will stand for the least upper bound \underline{in $\Aa_s$}
of suitable families of bounded operators.\par
For arbitrary positive $n$ and $m$ we have $0 \leqsl A_{\sigma} \cdotaff (Z_n Z_m') \leqsl A'
\cdotaff (Z_n Z_m') = m B' Z_n Z_m' \in \Aa_s$ which yields
\begin{multline*}
G \cdotaff (Z_n Z_m') = G_n Z_m' = [\sup_{\sigma\in\Sigma}{\!}^{\Aa} (A_{\sigma} \cdotaff Z_n)] Z_m'
= \sup_{\sigma\in\Sigma}{\!}^{\Aa} [(A_{\sigma} \cdotaff Z_n) Z_m']\\
= \sup_{\sigma\in\Sigma}{\!}^{\Aa} [A_{\sigma} \cdotaff (Z_n Z_m')] \leqsl A' \cdotaff (Z_n Z_m').
\end{multline*}
Now, as before, it suffices to note that $X = \sum_{n=1}^{\infty} \sum_{m=1}^{\infty} (X \cdotaff
(Z_n Z_m')) Z_n Z_m'$ for $X \in \{G,A'\}$ and then apply \LEM{ord}.
\end{proof}

The argument presented above contains a proof of the following convenient property.

\begin{cor}{=}
If $\TTt$ is an increasing net in $\Aa_s$ which is upper bounded in $\Aa_s$, then its least upper
bounds in $\Aa_s$ and $\Aff_s(\Aa)$ coincide.
\end{cor}

We need one more simple lemma.

\begin{lem}{decomp}
Let $T$ be any member of $\Aff(\Aa)$ and $\{Z_{\lambda}\}_{\lambda\in\Lambda}$ be a partition
in $\Aa$. If $T_{\lambda} \df T \cdotaff Z_{\lambda}$ is a bounded operator for any $\lambda \in
\Lambda$, then $T_{\lambda} Z_{\lambda} = T_{\lambda}$ for all $\lambda \in \Lambda$ and $T =
\sum_{\lambda\in\Lambda} T_{\lambda} Z_{\lambda}$.
\end{lem}
\begin{proof}
Since $T_{\lambda}$ is bounded, we get $T_{\lambda} Z_{\lambda} = T_{\lambda} \cdotaff Z_{\lambda}
= T \cdotaff Z_{\lambda} = T_{\lambda}$. Express $T$ in the form $T = \sum_{n=1}^{\infty} n S W_n$
with $S \in \Aa$ and $W_n \in \Zz(\Aa)$. Then $T_{\lambda} = \sum_{n=1}^{\infty} T_{\lambda} W_n =
\sum_{n=1}^{\infty} (T \cdotaff Z_{\lambda}) \cdotaff W_n = \sum_{n=1}^{\infty} (T \cdotaff W_n)
\cdotaff Z_{\lambda} = \sum_{n=1}^{\infty} n B (W_n Z_{\lambda})$ and hence
\begin{equation*}
\sum_{\lambda\in\Lambda} T_{\lambda} Z_{\lambda} = \sum_{\lambda\in\Lambda} \sum_{n=1}^{\infty}
nB (W_n Z_{\lambda}) = \sum_{n=1}^{\infty} \sum_{\lambda\in\Lambda} nB (W_n Z_{\lambda})
= \sum_{n=1}^{\infty} nB W_n = T
\end{equation*}
and we are done.
\end{proof}

Now we are ready to give

\begin{proof}[Proof of item \textup{(tr6)} in \THM{main}]
We already know from \LEM{sup} and (tr1) that both $A \df \sup_{\sigma\in\Sigma} A_{\sigma}$ and
$A' \df \sup_{\sigma\in\Sigma} \traff(A_{\sigma})$ are well defined. As in the proof of \LEM{sup},
we may and do assume that each operator $A_{\sigma}$ is non-negative. As usual, we express $A$
in the form $A = \sum_{n=1}^{\infty} n B Z_n$ where $B \in \Aa$. Then, from the very definition
of $\traff$ we deduce that $\traff(A) = \sum_{n=1}^{\infty} n \tr_{\Aa}(B) Z_n$. Further, the proof
of \LEM{sup}, combined with \COR{=}, yields that
\begin{equation*}
A = \sum_{n=1}^{\infty} [\sup_{\sigma\in\Sigma} (A_{\sigma} \cdotaff Z_n)] Z_n.
\end{equation*}
Consequently, $n B Z_n = A \cdotaff Z_n = \sup_{\sigma\in\Sigma} (A_{\sigma} \cdotaff Z_n)$. Now
the normality of $\tr_{\Aa}$ implies that $n \tr_{\Aa}(B) Z_n = \sup_{\sigma\in\Sigma}
\tr_{\Aa}(A_{\sigma} \cdotaff Z_n)$. But $\tr_{\Aa}(A_{\sigma} \cdotaff Z_n) = \traff(A_{\sigma})
\cdotaff Z_n$ (see (tr5)). We claim that $\sup_{\sigma\in\Sigma} (\traff(A_{\sigma}) \cdotaff Z_n)
= A' \cdotaff Z_n$. (To convince of that, first note that inequality `$\leqsl$' is immediate. To see
the reverse inequality, denote by $A_1'$ and $A_2'$, respectively, $\sup_{\sigma\in\Sigma}
(A_{\sigma} \cdotaff Z_n)$ and $\sup_{\sigma\in\Sigma} (A_{\sigma} \cdotaff (I - Z_n))$, and observe
that $A_{\sigma} \leqsl A_1' \plusaff A_2'$ and consequently $A' \leqsl A_1' \plusaff A_2'$, from
which one infers that $A' \cdotaff Z_n \leqsl A_1' \cdotaff Z_n \plusaff A_2' \cdotaff Z_n$, but
$A_2' \leqsl A' \cdotaff (I - Z_n)$ and thus $A_2' \cdotaff Z_n = 0$.) These observation lead us
to $A' \cdotaff Z_n = n \tr_{\Aa}(B) Z_n \in \Aa$. So, \LEM{decomp} yields $A' = \sum_{n=1}^{\infty}
n \tr_{\Aa}(B) Z_n = \traff(A)$.
\end{proof}

As we noted in the introductory part, condition (tr6) is a counterpart of normality
(in the terminology of Takesaki; see Definition~2.1 in Chapter~V of \cite{ta1}) of center-valued
traces in finite von Neumann algebras. Thus, the question of whether it is possible to equip
$\Aff(\Aa)$ with a topology (defined \textit{naturally}) with respect to which the center-valued
trace $\traff$ is continuous naturally arises. This will be a subject of further investigations.

\section{\texorpdfstring{Trace-like mappings in $\Aff(\Aa)$ and the type of $\Aa$}
 {Trace-like mappings in Aff(A) and the type of A}}

As \PRO{main} shows, finite type~I von Neumann algebras may be characterised (among all
von Neumann algebras) as those whose (full) sets of affiliated operators admit mappings which
resemble center-valued traces. The aim of the section is to prove \PRO{main}, which we now turn to.

\begin{proof}[Proof of \PRO{main}]
First observe that if $A \in \Aa$, then $\varphi(A)$ is bounded and consequently $\varphi(A) \in
\Aa$. Indeed, it is enough to show this for non-negative $A \in \Aa$. Such $A$ satisfies $\|A\|I-A
\geqsl 0$, therefore $\varphi(\|A\|I-A)$ and $\varphi(A)$ are non-negative (by (b)). But it follows
from (e) and (a) that
\begin{equation*}
\varphi(\|A\|I-A) = \varphi(\|A\|I) - \varphi(A) = \|A\|I - \varphi(A),
\end{equation*}
which means that $0 \leqsl \varphi(A) \leqsl \|A\|I$ and hence $\varphi(A)$ is bounded.
As $\varphi(A)$ commutes with each unitary operator in $\Aa$ (by (f)), we conclude that $\varphi(A)
\in \Zz(\Aa)$ for each $A \in \Aa$. So, $\psi = \varphi\bigr|_{\Aa}\dd \Aa \to \Zz(\Aa)$ is linear
(thanks to (a)) and satisfies all axioms of a center-valued trace (see (b), (d) and (e)), hence
$\Aa$ is finite. Note also that $\varphi(X) = \tr_{\Aa}(X)$ for each $X \in \Aa$.\par
Suppose that $\Aa$ is not type~I. Then one can find a non-zero projection $Z \in \Zz(\Aa)$ such
that $\Aa_0 \df \Aa Z$ is type~$\tII_1$. Recall that $\tr_{\Aa}\bigr|_{\Aa_0}$ coincides with
the center-valued trace $\tr_{\Aa_0}$ of $\Aa_0$.\par
Every type~$\tII_1$ von Neumann algebra $\Ww$ has the following property: for each projection $P \in
\Ww$ and an operator $C \in \Zz(\Ww)$ such that $0 \leqsl C \leqsl \tr_{\Ww}(P)$ there exists
a projection $Q \in \Ww$ for which $Q \leqsl P$ and $\tr_{\Ww}(Q) = C$ (to convince of that, see
Theorem~8.4.4 and item (vii) of Theorem~8.4.3, both in \cite{kr2}). Involving this property,
by induction we define a sequence $(P_n)_{n=1}^{\infty}$ of projections in $\Aa_0$ as follows:
$P_1 \in \Aa_0$ is arbitrary such that $\tr_{\Aa_0}(P_1) = \frac12 Z$; and for $n > 1$, $P_n \in
\Aa_0$ is such that $P_n \leqsl Z - \sum_{k=1}^{n-1} P_k$ and $\tr_{\Aa_0}(P_n) = \frac{1}{2^n} Z$.
Observe that the projections $P_n$, $n \geqsl 1$, are mutually orthogonal and for any $N \geqsl 1$,
\begin{equation*}
\tr_{\Aa}\Bigl(\sum_{k=1}^N 2^k P_k\Bigr) = \sum_{k=1}^N 2^k \tr_{\Aa_0}(P_k) = N Z.
\end{equation*}
Now for $N \geqsl 0$, put $T_N \df \sum_{k=N+1}^{\infty} 2^k P_k$ (the series understood pointwise,
similarly as in \DEF{series}). As each $P_k$ belongs to $\Aa$, we see that $T_N \in \Aff(\Aa_0)$.
What is more, $T_N$ is non-negative and we infer from axiom (c) that
\begin{equation*}
\varphi(T_0) = \varphi\Bigl(\sum_{k=1}^N 2^k P_k\Bigr) + \varphi(T_N) =
\tr_{\Aa}\Bigl(\sum_{k=1}^N 2^k P_k\Bigr) + \varphi(T_N) = N Z + \varphi(T_N).
\end{equation*}
Therefore, for $\xi \in \DDd(\varphi(T_0)) = \DDd(\varphi(T_N))$, we get:
\begin{equation*}
\scalar{\varphi(T_0)\xi}{\xi} = N \|Z \xi\|^2 + \scalar{\varphi(T_N)\xi}{\xi} \geqsl N \|Z \xi\|^2
\end{equation*}
(here $\scalarr$ denotes the inner product in $\HHh$). Since $N$ can be arbitrarily large, the above
implies that $Z \xi = 0$ for every $\xi \in \DDd(\varphi(T_0))$. But this is impossible, because
$Z \neq 0$ and $\DDd(\varphi(T_0))$ is dense in $\HHh$. The proof is complete.
\end{proof}

\end{document}